\documentclass[12pt]{amsart}
\usepackage{amssymb,latexsym,amsthm}
\usepackage{graphicx}

\newtheorem{theorem}{Theorem}[section]
\newtheorem{lemma}[theorem]{Lemma}

\newtheorem{pro}{Problem}

\theoremstyle{definition}

\newtheorem{example}[theorem]{Example}

\theoremstyle{remark}

\numberwithin{equation}{section}

\begin{document}

\baselineskip=17pt

\title[Isometries between groups]
{Isometries between groups of invertible elements in Banach algebras}

\author{Osamu~Hatori}
\address{Department of Mathematics, Faculty of Science, 
Niigata University, Niigata 950-2181 Japan}
\curraddr{}
\email{hatori@math.sc.niigata-u.ac.jp}

\thanks{The author was partly 
supported by the Grants-in-Aid for Scientific 
Research, The 
Ministry of Education, Science, Sports and Culture, Japan.}



\maketitle

\begin{abstract}
We show that if $T$ is an isometry (as metric spaces) from an 
open subgroup of the group of the invertible elements 
in a unital semisimple commutative Banach algebra 
onto an open subgroup of the group of the invertible elements 
in a unital Banach algebra, then 
$T(1)^{-1}T$ is an 
isometrical group isomorphism. In particular, 
$T(1)^{-1}T$ is extended to an isometrical real algebra 
isomorphism from $A$ onto $B$.
\end{abstract}
\section{Introduction}
A long tradition of inquiry seeks sufficient sets of conditions on 
(not only linear) isometries  between Banach algebras 
in order that they are algebraically isomorphic. 
The history of the problem probably dates back to 
a theorem of Banach \cite[Theorem XI. 3]{b}, which is 
the original form of 
the Banach-Stone theorem. 
The Banach-Stone theorem now 
states that two Banach spaces $C(X)$ and $C(Y)$ of 
the complex-valued continuous functions on compact Hausdorff spaces 
$X$ and $Y$ respectively are isomorphic as Banach space if and only if 
$X$ and $Y$ are homeomorphic to each other, therefore if and only if 
the two  are isomorphic as Banach algebras. One 
can say that the multiplication in the Banach algebra $C(X)$ is restored 
from the structure of a Banach space in the category of $C(K)$-spaces. 
Jarosz \cite{ja1} generalized the theorem in the sense 
that the multiplication in 
a uniform algebra is restored from the structure of a Banach space in 
the category of unital Banach algebras (cf. \cite{naga,ja2,ja3}).

In this paper we consider the problem of the same vine. 
Suppose that $B$ is a unital Banach algebra. 
The structure as a metrizable group of the 
group $B^{-1}$ of all the invertible elements in $B$ is said to be 
restored from the metric structure in the category of unital
Banach algebras if $B^{-1}$ is isometrically 
isomorphic as a metrizable 
group to $B_1^{-1}$ whenever $B_1$ is the unital Banach 
algebra and $B^{-1}$ is isometric to $B_1^{-1}$ as a metric space. 
In this paper we show that the structure as a metrizable group 
of the group of the invertible elements in 
the unital semisimple commutative Banach algebra is 
restored from the metric structure in the category of unital 
Banach algebras (Theorem \ref{main}). 
In this case the 
multiplication in the unital semisimple commutative Banach algebra is 
also restored.
It is compared with the following;
the multiplication in a certain 
unital semisimple commutative Banach algebra 
is {\it not} restored from the 
structure as a Banach space in the category of unital 
Banach algebras (see Example \ref{ww+}).

Throughout the paper we denote the unit element in a Banach algebra by $1$ and 
for a complex number $\lambda$, $\lambda 1$ is abbreviated by $\lambda$. 
The maximal ideal space of a unital semisimple commutative 
Banach algebra $A$ is denoted by $\Phi_A$.
We may suppose that $f\in A$ is a continuous function on $\Phi_A$ 
by identifying $f$ itself with its Gelfand transform;
the Gelfand transform of $f$ is also denoted by $f$, for simplicity.
The spectral radius of $f\in A$ is equal to 
the supremum norm of $f$ on $\Phi_A$ and is denoted by $\|f\|_{\infty}$.
\section{Lemmata}
Let ${\mathcal B}_1$ and ${\mathcal B}_2$ be real normed spaces.  
The theorem of Mazur and Ulam \cite{mu,v} states that 
if ${\mathcal B}_1$ is isometric to 
${\mathcal B}_2$ as a metric space, then they are isometrically 
isomorphic to each other as real normed spaces. 
Applying an idea of V\"ais\"al\"a \cite{v}, the following local version also 
holds.
\begin{lemma}\label{lmu}
Let ${\mathcal B}_1$ and ${\mathcal B}_2$ be real normed spaces, 
$U_1$ and $U_2$ non-empty open subsets of ${\mathcal B}_1$ and 
${\mathcal B}_2$ respectively. Suppose that ${\mathcal T}$ is 
a surjective isometry from $U_1$ onto $U_2$. 
If $f,g\in U_1$ satisfy that $(1-r)f+rg\in U_1$ for every 
$r$ with $0\le r \le 1$, then the equality
\[
{\mathcal T}(\frac{f+g}{2})=\frac{{\mathcal T}(f)+{\mathcal T}(g)}{2}
\]
holds.
\end{lemma}
\begin{proof}
Let $h,h'\in U_1$. Suppose that $\varepsilon >0$ satisfies that 
$\frac{\|h-h'\|}{2}<\varepsilon$, and 
\[
\{u\in B_1:\|u-h\|<\varepsilon,\,\,\|u-h'\|<\varepsilon \}\subset U_1,
\]
\[
\{a\in B_2:\|a-{\mathcal T}(h)\|<\varepsilon,\,\,
\|a-{\mathcal T}(h')\|<\varepsilon \}\subset U_2.
\]
We will show that 
${\mathcal T}(\frac{h+h'}{2})=\frac{{\mathcal T}(h)+{\mathcal T}(h')}{2}$.
Set 
$r=\frac{\|h-h'\|}{2}$ and let
\[
L_1=
\{u\in {\mathcal B}_1:\|u-h\|=r=\|u-h'\|\},
\]
\[
L_2=
\{a\in {\mathcal B}_2:\|a-{\mathcal T}(h)\|=r=\|a-{\mathcal T}(h')\|\}.
\]
Set also $c_1=\frac{h+h'}{2}$ and $c_2=\frac{{\mathcal T}(h)+
{\mathcal T}(h')}{2}$. Then we have ${\mathcal T}(L_1)=L_2$, 
$c_1\in L_1 \subset U_1$, and $c_2\in L_2 \subset U_2$. 
Let 
\[
\psi_1(x)=h+h'-x \quad (x\in {\mathcal B}_1)
\]
and 
\[
\psi_2(y)={\mathcal T}(h)+{\mathcal T}(h')-y \quad (y\in {\mathcal B}_2).
\]
Then we see that $\psi_1(c_1)=c_1$, $\psi_1(L_1)=L_1$, 
and $\psi_2(L_2)=L_2$. 
Let $Q=\psi_1\circ{\mathcal T}^{-1}\circ\psi_2\circ{\mathcal T}$. 
A simple calculation shows that  
\[
2\|w-c_1\|=\|\psi_1(w)-w\|,\quad (w\in L_1)
\]
and 
\[
\|\psi_1(z)-w\|=\|\psi_1\circ Q^{-1}(z)-Q(w)\|,\quad (z,w \in L_1)
\]
hold. 
Applying these equations we see that
\begin{multline*}
\|Q^{2^{k+1}}(c_1)-c_1\|=\|\psi_1\circ Q^{2^{k+1}}(c_1)-c_1\|
\\
=\|\psi_1\circ Q^{2^k}(c_1)-Q^{2^k}(c_1)\|
=2\|Q^{2^k}(c_1)-c_1\|
\end{multline*}
hold for every nonzero integer $k$, where $Q^{2^n}$ denotes the 
$2^n$-time composition of $Q$. 
By induction we see for every non-negative integer $n$ that 
\[
\|Q^{2^n}(c_1)-c_1\|=2^{n+1}\|c_2-{\mathcal T}(c_1)\|
\]
holds. 
Since $Q(L_1)=L_1$ and $L_1$ is bounded we see that $c_2=
{\mathcal T}(c_1)$, i.e., ${\mathcal T}(\frac{h+h'}{2})=
\frac{{\mathcal T}(h)+{\mathcal T}(h')}{2}$. 

We assume that $f$ and $g$ are as described. Let
\[
K=\{(1-r)f+rg:0\le r\le 1\}.
\]
Since $K$ and ${\mathcal T}(K)$ are compact, there is $\varepsilon >0$ with
\[
d(K,{\mathcal B}_1\setminus U_1)>\varepsilon, \quad 
d({\mathcal T}(K), {\mathcal B}_2\setminus U_2)>\varepsilon,
\]
where 
$d(\cdot , \cdot )$ denotes the distance of two sets.
Then for every $h\in K$ we have 
\[
\{u\in {\mathcal B}_1:\|u-h\|<\varepsilon\}\subset U_1
\]
and
\[
\{b\in {\mathcal B}_2:\|b-{\mathcal T}(h)\|<\varepsilon \}\subset U_2.
\]
Choose a natural number $n$ with $\frac{\|f-g\|}{2^n}<\varepsilon$. 
Let 
\[
h_k=\frac{k}{2^n}(g-f)+f
\]
for each $0\le k\le 2^n$. By the first part of the proof we have
\[
{\mathcal T}(h_k)+{\mathcal T}(h_{k+2})-2{\mathcal T}(h_{k+1})=0
\qquad \text{($k$)}
\]
holds for $0\le k\le 2^n-2$. For $0\le k\le 2^n-4$,
adding the equations ($k$), 2 times of ($k+1$), and ($k+2$) we have
\[
{\mathcal T}(h_k)+{\mathcal T}(h_{k+4})-2{\mathcal T}(h_{k+2})=0,
\]
whence the equality 
\[
{\mathcal T}(\frac{f+g}{2})=\frac{{\mathcal T}(f)+{\mathcal T}(g)}{2}
\]
holds by induction on $n$.
\end{proof}

Let $B$ be a unital Banach algebra. The exponential spectrum for $a\in B$ is 
\[
\sigma_{\exp B}(a)=\{\lambda \in {\mathbb C}:a-\lambda \not\in \exp B\},
\]
where $\exp B$ denotes the principal component of $B^{-1}$; $\exp B$ 
is the set of $\exp a$ for all $a\in B$ for the case where $B$ is 
commutative, and
$\exp B$ is the set of all the finite products of the form of 
$\exp a$ for $a\in B$ in 
general.
A complex-valued function $\varphi$ on $B$ is said to be a selection from the 
exponential spectrum if $\varphi (a)\in \sigma_{\exp B}(a)$ whenever 
$a\in B$. 

To prove Theorem \ref{main}, we apply a lemma concerning complex-linearity of 
real linear selections from the exponential spectrum, which is a version of a 
result due to Kowalski and S\l odkowski \cite[Lemma 2.1]{ks}. 
\begin{lemma}\label{ks}
Let $B$ be a unital Banach algebra. Suppose that 
$\varphi :B\to {\mathbb C}$ is a real linear selection from the exponential 
spectrum. Then $\varphi$ is a complex homomorphism.
\end{lemma}
\begin{proof}
A proof is similar to that for \cite[Lemma 2.1]{ks} in which the 
spectral maping theorem is applied. It is not suitable for the 
exponential spectrum, we apply an alternative. 
For $x\in B$, let
\[
\varphi_1(x)=\mathrm{Re}\varphi (x)-i\mathrm{Re}\varphi (ix)
\]
and
\[
\varphi_2(x)=\mathrm{Im}\varphi (ix)+i\mathrm{Im}\varphi (x).
\]
As in the same way as in the proof of \cite[Lemma 2.1]{ks}
$\varphi_1$ and $\varphi_2$ are complex linear selections from 
the exponential spectrum, hence they are complex homomorphisms 
by the
original proof of the Gleason-Kahane-\. Zelazko theorem \cite{kz,z} 
(cf. \cite{j4}).
We will show that $\varphi_1=\varphi_2$, which will 
force that $\varphi$ is a complex homomorphism. 
Soppose not. Then there is $a\in A$ with $\varphi_1(a)=1$ and 
$\varphi_2(a)=0$. Let 
\[
h(z)=\exp^{\frac{\pi i z}{2}}-1.
\]
Since $\varphi_1$ and $\varphi_2$ are continuous, we see that
\begin{equation*}\begin{split}
\varphi (h(a))&=\mathrm{Re}\varphi_1(h(a))+i\mathrm{Im}\varphi_2(h(a)) \\
&=\mathrm{Re}h(\varphi_1(a))+i\mathrm{Im}h(\varphi_2(a))=-1.
\end{split}\end{equation*}
Since $\varphi$ is a selection from the exponential spectrum, 
$-1\in \sigma_{\exp B}(h(a))$. On the other hand $h(a)+1\in \exp B$, so 
that $-1\not\in \sigma_{\exp B}(h(a))$, which is a contradiction proving that 
$\varphi_1=\varphi_2$.
\end{proof}

\section{Main results}
\begin{theorem}\label{submain}
Let $A$ be a unital semisimple commutative Banach algebra,
$B$ a unital Banach algebra, and 
\[
\Omega_A=\{f\in A: \text{$\|f-r\|<r$ for some positive real number $r$}\},
\]
\[
\Omega_B=\{a\in B: \text{$\|a-r\|<r$ for some positive real number $r$}\}.
\]
Suppose that $U$ is an open set such that $\Omega_A\subset U\subset A^{-1}$ 
and $({\mathbb C}\setminus \{0\})U\subset U$, 
and  that $V$ is an open set such that $\Omega_B\subset V\subset B^{-1}$, 
$({\mathbb C}\setminus \{0\})V\subset V$, and $V\Omega_B\subset V$.
Let $g\in B^{-1}$. If $T$ is an surjective isometry from $U$ onto 
$gV$, then $T$ is extended to a real linear isometry from $A$ onto $B$.
\end{theorem}
\begin{proof}
Applying Lemma \ref{lmu} we see that the equality 
\[
T\left(\frac{f+g}{2}\right)=\frac{T(f)+T(g)}{2}
\]
holds for every pair $f$ and $g$ in $\Omega_A$ since 
$\Omega_A$ is convex.

We will show that 
$
\lim_{U\ni f\to 0}T(f)=0.
$
Since $T^{-1}$ is an isometry $\lim_{gV\ni a\to 0}T^{-1}(a)$ exists 
by a routine argument on Cauchy sequences.
Let $u=\lim_{gV\ni a\to 0}T^{-1}(a)$. 
We see that $\sigma (u)=\{0\}$. (Suppose not; 
$0\ne \lambda \in \sigma (u)$. Then $-\lambda \in U$ since 
$|\lambda|\in \Omega_A$ and $({\mathbb C}\setminus\{0\})U\subset U$. 
Let $T(-\lambda)=c_{\lambda}\in gV$. The inequality $(1-s)(1-r)+sr>0$ 
holds for every $0<r<1$ 
and $0\le s\le 1$, hence
\[
(1-s)\{(1-r)c_{\lambda}\}+src_{\lambda}\in gV.
\]
Applying Lemma \ref{lmu} with $f=(1-r)c_{\lambda}$, $g=rc_{\lambda}$ 
we have
\[
T^{-1}\left(\frac{c_{\lambda}}{2}\right)=
T^{-1}\left(\frac{(1-r)c_{\lambda}+rc_{\lambda}}{2}\right)=
\frac{T^{-1}((1-r)c_{\lambda})+T^{-1}(rc_{\lambda})}{2}.
\]
Letting $r\to 0$ we have
\[
T^{-1}\left(\frac{c_{\lambda}}{2}\right)=\frac{-\lambda +u}{2},
\]
which is a contradiction since 
$T^{-1}\left(\frac{c_{\lambda}}{2}\right) \in U\subset A^{-1}$ 
and 
$\frac{-\lambda + u}{2}\not\in A^{-1}$ for $\lambda \in \sigma (u)$.) 
Since $A$ is semisimple and commutative, we see that $u=0$;
$\lim_{gV\ni a\to 0}T^{-1}(a)=0$. It turns out that 
\begin{equation}\label{limTf}
\lim_{U\ni f\to 0}T(f)=0
\end{equation}
since $T$ is isometry.

Next we will show that $T(-f)=-T(f)$ for every $f\in U$. 
Let $f\in U$. Then $-f\in U$, and for every integer $n$, 
$-f+\frac{i}{n}f\in U$. 
We also see  
\[
(1-r)f+r(-f+\frac{i}{n}f) \in U
\]
for every $0\le r \le 1$ and every  integer $n$. 
Then by Lemma \ref{lmu} 
\[
T\left( \frac{i}{2n}f\right)=T\left(\frac{f+(-f+\frac{i}{n}f)}{2}\right)
=\frac{T(f)+T(-f+\frac{i}{n}f)}{2}
\]
hold. Letting $n\to \infty$ we have $T(-f)=-T(f)$ by (\ref{limTf}).

Next we will show that 
\begin{equation}\label{12}
T\left(\frac{f}{2}\right)=\frac{T(f)}{2}
\end{equation}
holds 
for every $f\in U$. 
Let $f\in U$. 
Then for every $1>\varepsilon >0$ and every $0\le r \le 1$
\[
(1-r)f +r\varepsilon f \in U.
\]
Hence 
$T\left(\frac{f+\varepsilon f}{2}\right)=\frac{T(f)+T(\varepsilon f)}{2}$ 
holds by Lemma \ref{lmu}, then letting $\varepsilon \to 0$ 
the equation (\ref{12}) holds.

Let $f\in U$. 
Suppose that $T(kf)=kT(f)$ holds for a positive integer $k$. 
Then 
\[
T\left( \frac{f+kf}{2}\right) =\frac{T(f)+T(kf)}{2}=\frac{(k+1)T(f)}{2}
\]
and by 
(\ref{12}) 
\[
T\left(\frac{f+kf}{2}\right)=\frac{T((k+1)f)}{2}
\]
hence by induction 
$T(nf)=nT(f)$ holds for every positive integer $n$. 
Then for any pair of positive integers $m$ and $n$, 
\[
mT(\frac{n}{m}f)=T(m\frac{n}{m}f)=T(nf)=nT(f)
\]
holds, hence $T(\frac{n}{m}f)=\frac{n}{m}T(f)$ holds. 
By continuity of $T$, 
$T(rf)=rT(f)$ 
holds for every $f\in U$ and $r>0$. 
Henceforce  
\begin{equation}\label{3.15}
T(rf)=rT(f)
\end{equation}
holds for every $f\in U$ and for a non-zero real number $r$ 
since $T(-f)=-T(f)$.

Applying Lemma \ref{lmu} and (\ref{12}) we see that
\begin{equation}\label{sum}
T(f+g)=T(f)+T(g)
\end{equation}
holds for every pair $f$ and $g$ in $U$ whenever 
$(1-r)f+rg\in U$ holds for every $0\le r \le 1$. In particular 
(\ref{sum}) holds if $f,g\in \Omega_A$. 

Define the map $T_U:A\to B$ by $T_U(0)=0$ and 
\[
T_U(f)=T(f+2\|f\|)-T(2\|f\|)
\]
for a non-zero $f\in A$. The map $T_U$ is well-defined since 
$f+2\|f\|$ and $2\|f\|$ are in $\Omega_A$ for every non-zero $f\in A$ and 
$T$ is defined on $U\supset \Omega_A$. 
If, in particular, $f\in \Omega_A$, then 
$T(f+2\|f\|)=T(f)+T(2\|f\|)$ holds, so that $T_U(f)=T(f)$ holds. 

We will show that $T_U$ is real-linear. Let $f\in A\setminus \{0\}$. Then 
$f+r\in \Omega_A$ for every $r\ge 2\|f\|$, whence by (\ref{sum})
\[
T(f+2\|f\|)+T(r)=T(f+2\|f\|+r)=T(f+r)+T(2\|f\|),
\]
so that
\begin{equation}\label{UOmega}
T_U(f)=T(f+r)-T(r)
\end{equation}
holds for every $r\ge 2\|f\|$. Let $f,g \in A$. 
Then $T_U(f+g)=T_U(f)+T_U(g)$ holds if $f=0$ or $g=0$. Suppose that 
$f\ne 0$ and $g\ne 0$. Then by (\ref{sum}) and (\ref{UOmega}) we have
\begin{equation*}\begin{split}
T_U(f+g)&=
 T(f+g+2\|f\|+2\|g\|)-T(2\|f\|+2\|g\|) \\
&= T(f+2\|f\|)+T(g+2\|g\|)-T(2\|f\|)-T(2\|g\|) \\
&= T_U(f)+T_U(g)
\end{split}
\end{equation*}
holds. If $f=0$ or $r=0$ then $T_U(rf)=rT_U(f)$. Suppose that 
$f\ne 0$ and $r\ne 0$. If $r>0$, then by (\ref{3.15})
\begin{equation*}\begin{split}
T_U(rf) &=
T(rf+2\|rf\|)-T(2\|rf\|) \\
&= T(r(f+2\|f\|))-T(r2\|f\|) \\
&=rT(f+2\|f\|)-rT(2\|f\|)=rT_U(f)
\end{split}
\end{equation*}
If $r<0$, then
\[
T_U(rf)=(-r)\left(T(-f+2\|f\|)-T(2\|f\|)\right).
\]
Since $-f+2\|f\|$, $f+2\|f\|\in \Omega_A$ we have
\[
T(-f+2\|f\|)-T(2\|f\|)=-T(f+2\|f\|)+T(2\|f\|).
\]
It follows that 
\[
T_U(rf)=(-r)\left(-T(f+2\|f\|)+T(2\|f\|)\right)=rT_U(f).
\]

We will show that $T_U$ is surjective.
Let $a\in B$.  
Then 
\[
(T(1))^{-1}a+r\in \Omega_B\subset V,
\]
so 
\[
a+T(r)=a+rT(1)\in T(1)\Omega_B \subset gV\Omega_B\subset gV
\]
holds whenever $\|(T(1))^{-1}a\|<r$ and $\|a\|<r$ 
for $T(1)\in gV$. We also have 
\[
\|T^{-1}(a+T(r))-r\|=\|a+T(r)-T(r)\|<r,
\]
thus $T^{-1}(a+T(r))\in \Omega_A$. Let $f= T^{-1}(a+T(r))-r\in A$. 
Then 
$f+r=T^{-1}(a+T(r))\in \Omega_A$. Hence by (\ref{sum}) we see that 
\[
T(f+r)+T(2\|f\|)=T(f+2\|f\|+r)=T(f+2\|f\|)+T(r),
\]
so we have 
\[
a=T(f+r)-T(r)=T(f+2\|f\|)-T(2\|f\|)=T_U(f).
\]

We will show that $T_U$ is an isometry. Since $T_U$ is linear, 
it is sufficient to show that 
$\|T_U(f)\|=\|f\|$ for every 
$f\in A$. If $f=0$, the equation clearly holds. Suppose 
that$f\ne 0$. Then 
\[
\|T_U(f)\|=\|T(f+2\|f\|)-T(2\|f\|)\|=\|f+2\|f\|-2\|f\|\|=\|f\|
\]
hold.

We will show that $T_U$ is an extension of $T$, i.e., $T_U(f)=T(f)$ for 
every $f\in U$. Put 
$P=T_U^{-1}\circ T:U\to A$. Let $f\in U$. Then 
$P(f+2\|f\|)=f+2\|f\|$ holds since $f+2\|f\|\in \Omega_A$ and 
$T=T_U$ on $\Omega_A$. Thus we have
\begin{equation*}\begin{split}
2\|f\|&=\|T(f+2\|f\|)-T(f)\|\\
&=\|f+2\|f\|-P(f)\|\ge \|P(f)-f-2\|f\|\|_{\infty},
\end{split}
\end{equation*}
so that the range of $P(f)-f$ on the maximal ideal space $\Phi_A$ is contained 
in the closed disk in the complex plane 
with the radius $2\|f\|$ and the center $2\|f\|$. 
Applying $P(-f+2\|f\|)=-f+2\|f\|$ in the same way
\[
2\|f\|\ge \|P(f)-f+2\|f\|\|_{\infty}
\]
holds since $T(-f)=-T(f)$, 
so that the range of $P(f)-f$ is in the closed unit disk with the 
radius $2\|f\|$ and the center $-2\|f\|$. It follows that 
$\sigma (P(f)-f)=(P(f)-f)(\Phi_A)=\{0\}$. 
Since $A$ is semisimple and commutative, we see that 
$P(f)=f$ for every $f\in U$; $T_U(f)=T(f)$ holds for every $f\in U$.
\end{proof}

Two unital semisimple commutative Banach algebras which are 
isometrically isomorphic to each other 
as Banach spaces need not be isometrically isomorphic to each other
as Banach algebras.
\begin{example}\label{ww+}
Let $W$ be the Wiener algebra; 
\[
W=\{f\in C({\mathbb T}): \|f\|=\sum_{-\infty}^{\infty}|\hat f(n)|<\infty\},
\]
where ${\mathbb T}$ is the unit circle in the complex plane and 
$\hat f(n)$ denotes the $n$-th Fouriere coeficient and 
\[
W_+=\{f\in W: \text{$\hat f(n)=0$ for every $n < 0$}\}.
\]
Then 
\[
(T_W(f))(e^{i\theta})=\sum_{n=0}^{\infty}\hat f(n)e^{2ni\theta} +
\sum_{n=1}^{\infty}\hat f(-n)e^{(2n-1)i\theta}
\]
defines an isometric isomorphism as Banach spaces from 
$W$ onto $W_+$. On the other hand, $W$ is not isomorphic as 
a complex algebra to $W_+$ since the maximal ideal  
space of $W$ is ${\mathbb T}$ 
and that of $W_+$ is the closed unit disk, which is not homeomorphic to 
${\mathbb T}$.
\end{example}
Despite above isometries between two groups of invertible 
elements in unital semisimple commutative 
Banach algebras induce isometrical group isomorphisms.
\begin{theorem}\label{main}
Let $A$ be a unital semisimple commutative Banach algebra and $B$ 
a unital Banach algebra. Suppose ${\mathfrak A}$ and ${\mathfrak B}$
are open subgroups of $A^{-1}$ and $B^{-1}$ respectively. 
Supposet that $T$ is a surjective isometry (as a map between 
metric spaces) 
from ${\mathfrak A}$ onto
${\mathfrak B}$. Then $B$ is a semisimple and commutative, and 
$(T(1))^{-1}T$ is extended to an isometrical real algebra 
isomorphism from $A$ onto $B$. In particular, 
$A^{-1}$ is isometrically isomorphic to 
$B^{-1}$ as a metrizable group.
\end{theorem}
\begin{proof}
Since ${\mathfrak A}$ (resp. ${\mathfrak B}$) is an open subgroup of 
$A^{-1}$ (resp. $B^{-1}$), $\exp A\subset {\mathfrak A}$ 
(resp. $\exp B \subset {\mathfrak B}$), 
whence $\Omega_A\subset {\mathfrak A}$ 
(resp. $\Omega_B\subset {\mathfrak B}$). 
Applying Theorem \ref{submain} with $U={\mathfrak A}$, 
$V={\mathfrak B}$, and $g=1$, we obtain 
a surjective real linear isometry $T_U$ from $A$ onto $B$ which 
is the extension of $T$. 

We will show that $|T_U^{-1}(1)|=1$ on $\Phi_A$. Since $T_U^{-1}$ is a linear 
isometry, $\|T_U^{-1}(1)\|=1$ holds, hence $|T_U^{-1}(1)|\le 1$ on 
$\Phi_A$. Suppose that there exists 
$x\in \Phi_A$ such that $|T_U^{-1}(1)(x)|<1$. Since $T_U$ is an isometry
\begin{equation*}\begin{split}
1>|T_U^{-1}(1)(x)|
&=\|T_U^{-1}(1)-\left(T_U^{-1}(1)-(T_U^{-1}(1))(x)\right)\| \\
&=\|1-T_U\left(T_U^{-1}(1)-(T_U^{-1}(1))(x)\right)\|,
\end{split}
\end{equation*}
so that $T_U\left(T_U^{-1}(1)-(T_U^{-1}(1))(x)\right) \in \exp B \subset 
{\mathfrak B}$. Since $T_U=T$ on ${\mathfrak B}$ we have that 
\[
T_U^{-1}(1)-(T_U^{-1}(1))(x)\in {\mathfrak A}\subset A^{-1},
\]
which is a contradiction since 
$\left(T_U^{-1}(1)-(T_U^{-1}(1))(x)\right)(x)=0$. Henceforce  
\begin{equation}\label{TU}
\text{$|T_U^{-1}(1)|=1$ on $\Phi_A$}.
\end{equation}
holds.

In a way similar to the above we have
\begin{equation}\label{TUi}
\text{$|T_U^{-1}(i)|=1$ on $\Phi_A$}.
\end{equation}

Define $S:B\to A$ by $S(a)=\left(T^{-1}(1)\right)^{-1}T_U^{-1}(a)$ 
for $a \in B$. 
Then $S$ is a bounded real linear bijection from $B$ onto $A$ such that 
$S({\mathfrak B})={\mathfrak A}$. Let $a\in B$. Then 
\begin{equation}\label{sa}
\|S(a)\|=\|T^{-1}(1)\|\|\left(T^{-1}(1)\right)^{-1}T_U^{-1}(a)\|\ge 
\|T_U^{-1}(a)\|=\|a\|.
\end{equation}

Next we will show that $\left(S(i)\right)(\Phi_A)\subset i{\mathbb R}$. Let 
$x\in \Phi_A$. For every $r>0$ we see that
\begin{equation*}\begin{split}
|r\pm (S(i))(x)|&=|r(T^{-1}(1))(x)\pm (T_U^{-1}(i))(x)| \\
&=
|(T_U^{-1}(r\pm i))(x)|\le \|T_U^{-1}(r\pm i)\|=|r\pm i|
\end{split}
\end{equation*}
since $T_U^{-1}$ is real-linear and (\ref{TU}) holds, hence 
we have that $\left(S(i)\right)(\Phi_A)\subset i{\mathbb R}$, so that
\[
(S(i))(\Phi_A)\subset \{i,-i\}
\]
holds by (\ref{TU}) and (\ref{TUi}).

Let 
\[
\Phi_{A+}=\{x\in \Phi_A:S(i)(x)=i\},
\]
\[
\Phi_{A-}=\{x\in \Phi_A:S(i)(x)=-i\}.
\]
Then $\Phi_{A+}$ and $\Phi_{A-}$ are (possibly empty) closed and 
open subsets of 
$\Phi_A$ respectively and 
\[
\Phi_A=\Phi_{A+}\cup \Phi_{A-},\quad \Phi_{A+}\cap \Phi_{A-}=\emptyset.
\]
Define a function $\iota :C(\Phi_A)\to C(\Phi_A)$ by
\begin{equation*}
(\iota (f))(x)=
\begin{cases}
f(x), x\in \Phi_{A+} \\
\overline{f(x)}, x\in \Phi_{A-}
\end{cases}
\end{equation*}
Then $\iota$ is a real-linear bijection. Note that 
$\iota (S(i))=i$ and $\iota (A)$ is a complex algebra. 
Define the norm $\|\cdot \|_{\pm}$ 
on $\iota (A)$ by 
\[
\|\iota (f)\|_{\pm}=\max \{\|f\|_+,\|f\|_-\},
\]
where 
\[
\|f\|_+=\inf\{\|g\|:g\in A, \text{$g=f$ on $\Phi_{A+}$}\},
\]
\[
\|f\|_-=\inf\{\|h\|:h\in A, \text{$h=f$ on $\Phi_{A-}$}\}.
\]
Applying the \v Silov idempotent theorem and by a routine argument 
we see that $\iota (A)$ is a unital semisimple 
commutative Banach algebra with respect to the norm $\|\cdot \|_{\pm}$. 
Define $\tilde S:B\to \iota (A)$ by $\tilde S(a)=\iota (S(a))$ for 
$a\in B$. 
Then $\tilde S$ is a bounded real linear bijection from $B$ onto $A$ such that 
$\tilde S(1)=1$ and $\tilde S(i)=i$. 

Let $\phi \in \Phi_{\iota (A)}$. 
We will show that $\phi\circ \tilde S$ is a real linear selection from the 
exponential spectrum $\sigma_{\exp B}$, where 
\[
\sigma_{\exp B}(a)=\{\lambda \in {\mathbb C}: a-\lambda \not\in \exp B\}
\]
for $a\in B$. We only need to show that 
$\phi \circ \tilde S(a)\in \sigma_{\exp B}(a)$ for every $a\in B$. Let 
$a\in B$ and put $\lambda = \phi \circ \tilde S(a)$. Then 
$\tilde S(a)-\lambda \not\in (\iota (A))^{-1}$ since 
$\phi \in \Phi_{\iota(A)}$. Suppose that 
$\lambda \not\in \sigma_{\exp B}(a)$. Then 
\[
\tilde S(a-\lambda)\in \tilde S(\exp B)\subset \iota({\mathfrak A})
\subset \iota(A^{-1}).
\]
Note that $\iota(A^{-1})=(\iota(A))^{-1}$ holds. 
Since $\tilde S(1)=1$, $\tilde S(i)=i$, and $\tilde S$ is real-linear, 
\[
\tilde S(a-\lambda)=\tilde S(a)-\lambda,
\]
so
\[
\tilde S(a)-\lambda \in (\iota(A))^{-1}
\]
holds, which is a contradiction.

By Lemma \ref{ks} we see that $\phi\circ \tilde S$ is a complex homomorphism. 
It follows that $\tilde S$ is a (complex) algebra isomorphism from $B$ onto 
$\iota (A)$. In particular, we see that $B$ is semisimple and 
commutative. 

We see that $S=\left(T^{-1}(1)\right)^{-1}T_U^{-1}$ 
is a real algebra isomorphism 
from $B$ onto $A$. Since $B$ is semisimple and commutative, we see in a 
way similar to the above that
$\left(T(1)\right)^{-1}T_U$ 
is a real algebra isomorphism from $A$ onto $B$ such that 
\begin{equation}\label{TUT1}
\|\left(T(1)\right)^{-1}T_U(f)\|\ge \|f\|
\end{equation}
holds for every $f\in A$.

We will show that 
\begin{equation}\label{-}
\left(\left(T^{-1}(1)\right)^{-1}T_U^{-1}\right)^{-1}=(T(1))^{-1}T_U.
\end{equation}
Let $f\in A$ and put
\[
a=\left(\left(T^{-1}(1)\right)^{-1}T_U^{-1}\right)^{-1}(f).
\]
Then $a=T_U(T^{-1}(1)f)$ holds. On the other hand, since 
$\left(T^{-1}(1)\right)^{-1}T_U^{-1}$ is multiplicative, we see that
\begin{multline*}
T(1)T_U(T^{-1}(1)f)
= T_U\left(T^{-1}(1)(T^{-1}(1))^{-1}\right)
T_U(T^{-1}(1)f) \\
=\left(\left(T^{-1}(1)\right)^{-1}T_U^{-1}\right)^{-1}
\left(\left(T^{-1}(1)\right)^{-1}\right)
\left(\left(T^{-1}(1)\right)^{-1}T_U^{-1}\right)^{-1}(f)\\
=
\left(\left(T^{-1}(1)\right)^{-1}T_U^{-1}\right)^{-1}
\left(\left(T^{-1}(1)\right)^{-1}f\right) \\
=T_U\left(T^{-1}(1)(T^{-1}(1))^{-1}f\right)
=T_U(f).
\end{multline*}
Henceforce 
\[
\left(T(1)\right)^{-1}T_U(f)=T_U(T^{-1}(1)f)=
\left(\left(T^{-1}(1)\right)^{-1}T_U^{-1}\right)^{-1}(f)
\]
holds for every $f\in A$; (\ref{-}) holds.
Then by (\ref{sa}) and (\ref{TUT1}) we see that
\[
\|(T(1))^{-1}T_U(f)\|=\|f\|
\]
holds for every $f\in A$. We see that $(T(1))^{-1}T_U$ is an 
isometrica real algebra isomorphism from $A$ onto $B$, hence 
$(T(1))^{-1}T_U(A^{-1})=B^{-1}$, and we see that $A^{-1}$ is isometrically 
isomorphic to $B^{-1}$ as a metrizable group.
\end{proof}

We see by Theorem \ref{main} that the structure as a 
metrizable group of the group of the invertible elements in 
the unital semisimple commutative Banach algebra is 
restored from the metric structure of the group in the category of unital 
Banach algebras. 
\begin{pro}
In which unital Banach algebra is the structure as the metrizable 
group of the group of the invertible elements 
restored from the metric structure of the group 
in the category of unital Banach 
algebras?
\end{pro}
Theorem \ref{main} does not hold if $A$ is commutative but not semisimple  
as the following example shows.
\begin{example}\label{dame}
Let 
\[
A_0=\{
\left(
\begin{smallmatrix}
0&a&b \\0&0&c \\ 0&0&0
\end{smallmatrix}
\right)
:a,\,\,b,\,\,c\in {\mathbb C}\}.
\]
Let 
\[
A=\{
\left(
\begin{smallmatrix}
\alpha&a&b \\0&\alpha&c \\ 0&0&\alpha
\end{smallmatrix}
\right)
:\alpha,\,\,a,\,\,b,\,\,c\in {\mathbb C}\}
\]
be the unitization of $A_0$, where the multiplication (in $A_0$) is the 
zero multiplication; $MN=0$ for every $M,N\in A_0$.
Let $B=A$ as sets, while the multiplication in $B$ is the usual multiplication 
for matrixes. 
Then $A$ and $B$ are unital Banach algebras under the usual operator norm. 
Note that $A$ is commutative, but not semisimple. Note also that 
$A^{-1}=\{\left(
\begin{smallmatrix}
\alpha&a&b \\0&\alpha&c \\ 0&0&\alpha
\end{smallmatrix}
\right)\in A:\alpha\ne 0\}$ and $B^{-1}=\{\left(
\begin{smallmatrix}
\alpha&a&b \\0&\alpha&c \\ 0&0&\alpha
\end{smallmatrix}
\right)\in B:\alpha\ne 0\}$.
Put 
$F=\left(
\begin{smallmatrix}
0&0&7 \\0&0&0 \\ 0&0&0
\end{smallmatrix}
\right)$
Define 
$T:A^{-1}\to B^{-1}$ by $T(M)=M+F$. 
Then $T$ is well-defined and surjective (affine) isometry from 
$A$ onto $B$.
On the other hand $A^{-1}$ is not (group) isomorphic to $B^{-1}$
\end{example}

\end{document}